\documentclass[a4paper,12pt]{amsart} 

\usepackage{amssymb,amsmath}

\usepackage{color}

\newtheorem{thm}{Theorem}[section] 
\newtheorem{prop}[thm]{Proposition}
\newtheorem{cor}[thm]{Corollary} 
\newtheorem{lem}[thm]{Lemma}

\theoremstyle{definition}

\numberwithin{equation}{section}

\begin{document}

\title[Stable foliations with respect to Fuglede modulus]{Stable foliations with respect to Fuglede modulus and level sets of $p$--harmonic functions}
\author{Ma\l gorzata Ciska--Niedzia\l omska \and Kamil Niedzia\l omski}
\date{}

\begin{abstract}
We continue the study of the variation of the $p$--modulus of a foliation initiated by the first author. We derive the formula for the second variation which allows to study $p$--stable foliations. We obtain some results concerning codimension one $p$--stable foliations. Moreover, we derive the equation for the critical point of the $p$--modulus functional of a foliation given by the level sets of smooth function. We show the correlation with the $q$--harmonicity. We give some examples. In particular, we show that foliations given by the distance function are critical points of $p$--modulus functional.
\end{abstract}


\address{
Department of Mathematics and Computer Science \endgraf
University of \L\'{o}d\'{z} \endgraf
ul. Banacha 22, 90-238 \L\'{o}d\'{z} \endgraf
Poland
}
\email{mciska@math.uni.lodz.pl}
\email{kamiln@math.uni.lodz.pl}

\maketitle

\section{Introduction}

Extremal length \cite{ab}, the notion which was the starting point to the definition of modulus, is a conformal invariant widely used in the context of quasi--conformal mappings on the plane. The $p$--modulus of family of curves plays major role in many areas of mathematics, such as potential theory or geometric measure theory. In the latter, it allows to deal with Sobolev spaces on metric measure spaces.

A $p$--modulus was generalized to a family of measures and, in particular to a family of Lipschitz surfaces by Fuglede \cite{bf}. This leads in a natural way to a $p$--modulus of families of submanifolds on a Riemannian manifold and to the $p$--modulus of a foliation. Therefore it is natural to study the properties of such notion. 

The modulus of a family of hyper--surfaces separating two components of the complement of doubly connected set is the reciprocal of the capacity of this set \cite{wz,jh}. The capacity is an important tool in the potential theory, differential equations but has also implications to differential geometry \cite{jf0,jf,cm}. By the monotonicity of modulus, it is natural to find and describe these which modulus attains its maximum.  

In this paper, we continue the approach initiated in \cite{mc2} by the first author. We derive the second variation of $p$--modulus functional. This leads to the definition of $p$--stable foliations (for the $p$--modulus). We give more explicit condition for $p$--stable foliations in the codimension one case. This condition reminds the weighted Hardy type inequality, with the weight being the $p$--th power of the extremal function for the $p$--modulus.

We also focus on the foliations given by the level sets of any smooth functions. We derive the differential equation for critical points of $p$--modulus functional. We show the correlation of the critical points of $p$--modulus functional with the $q$--harmonicity of the function defining the foliation. We show that foliations given by the distance function are $p$--stable.

The main tools of all considerations are the following. 

In the context of the second variation of $p$--modulus we use the integral formula obtained by the first author \cite{mc1} and slightly generalized in this paper to the following
\begin{equation*}
\int_M f_0^{p-1}\varphi\hat{\psi}\,d\bar{\mu}=\int_M f_0^{p-1}\hat{\varphi}\psi\,d\bar{\mu},
\end{equation*} 
where $\bar{\mu}$ is the Lebesgue measure on $M$, $f_0$ the extremal function for the $p$--modulus of a given foliation $\mathcal{F}$ and $\hat{\varphi}$ is the integral over the leaves of $\mathcal{F}$ (for the definitions see the next section). This formula is of the independent interest and can be seen as a generalization of co--area formula.

In the context of $q$--harmonicity of a function $u:M\to\mathbb{R}$ on a Riemannian manifold and $p$--modulus of related foliation $\mathcal{F}_u$, we derive properties of the function
\begin{equation*}
\nu^q_u(t)=\int_{u^{-1}(t)}|\nabla u|^{q-1}\,d\mu_{u^{-1}(t)},\quad t\in u(M).
\end{equation*} 
We compute the differential of $\nu^q_u$ and use obtained formula in order to obtain the differential equation for the critical point of the $p$--modulus functional.

\section{Basic definitions and facts}

Let $(M,g)$ be a Riemannian manifold, $\mathcal{F}$ a $k$--dimensional foliation on $M$. Denote by $\bar{\mu}$ the Lebesgue measure on $M$ and let $\mu_L$ denotes the Lebesgue measure on the leaf $L\in\mathcal{F}$. 

Fix $p>1$. We say that a nonnegative Borel function $f:M\to\mathbb{R}$ is {\it admissible} if 
\begin{equation*}
\int_L f\,d\mu_L\geq 1
\end{equation*} 
for all $L\in\mathcal{F}$. The set of admissible functions is denoted by ${\rm adm}(\mathcal{F})$. The $p$--{\it modulus} of $\mathcal{F}$ is the number
\begin{equation*}
{\rm mod}_p(\mathcal{F})=\inf_{f\in{\rm adm}_{\mathcal{F}}}\int_M f^p\,d\bar{\mu}
\end{equation*}
in the case when ${\rm adm}(\mathcal{F})$ is nonempty, and we put ${\rm mod}_p(\mathcal{F})=\infty$ otherwise. Notice that $p$--modulus can be defined for any family $\mathcal{F}$ of $k$--dimensional submanifolds. The original definition of modulus was introduced for the family of measures by Fuglede \cite{bf}. The $p$--modulus has the following basic properties \cite{bf,bl1}
\begin{enumerate}
\item if $\mathcal{F}\subset\mathcal{G}$, then ${\rm mod}_p(\mathcal{F})\leq{\rm mod}_p(\mathcal{G})$,
\item if $\mathcal{F}\subset\bigcup_i\mathcal{F}_i$, then ${\rm mod}_p(\mathcal{F})\leq\sum_i{\rm mod}_p(\mathcal{F}_i)$.
\end{enumerate}
In other words, $p$--modulus is an outer measure (on the set of $k$--dimensional submanifolds).

If there is an admissible function $f_0$, which realizes the $p$--modulus, i.e.
\begin{equation*}
{\rm mod}_p(\mathcal{F})=\int_M f_0^p\,d\bar{\mu},
\end{equation*} 
then we call $f_0$ the {\it extremal function} for the $p$--modulus of $\mathcal{F}$. One can show the following characterization of existence of extremal function \cite{mc1}.

\begin{prop}
There is extremal function for the $p$--modulus of $\mathcal{F}$ if and only if for any subfamily $\mathcal{E}\subset\mathcal{F}$ such that ${\rm mod}_p(\mathcal{E})=0$ we have $\bar{\mu}(\bigcup\mathcal{E})=0$.
\end{prop}

Moreover extremal function has two important properties \cite{mc1}
\begin{enumerate}
\item $f_0>0$ almost everywhere,
\item $\int_L f_0\,d\mu_L=1$ for almost every leaf $L\in\mathcal{F}$.
\end{enumerate}

In the case, when a foliation $\mathcal{F}$ is given by the level sets of a submersion $\Phi:M\to N$, the extremal function for the $p$--modulus of $\mathcal{F}$, if exists, equals \cite{kp}
\begin{equation}\label{eq:extfunsubmersion}
f_0=\frac{(J\Phi)^{q-1}}{\int_{\Phi^{-1}\circ\Phi}(J\Phi)^{q-1}\,d\mu_{\Phi^{-1}\circ\Phi}}
\end{equation}
and
\begin{equation}\label{eq:modsubmersion}
{\rm mod}_p(\mathcal{F})=\int_N \left( \int_{\Phi^{-1}(y)} (J\Phi)^{q-1}\,d\mu_{\Phi^{-1}(y)}\right)^{1-p}\,d\mu_N(y),
\end{equation}
where $p$ and $q$ are conjugate coefficients, $p+q=pq$.

Now we will slightly improve some of the results in \cite{mc1}. Let us begin with some notation and facts. We say that a family $\mathcal{E}\subset\mathcal{F}$ is $p$--{\it exceptional} if its $p$--modulus is equal to $0$. Moreover, we say that a property $(P)$ holds $p$--almost everywhere (with respect to $\mathcal{F}$) if there is $p$--exceptional family $\mathcal{E}\subset\mathcal{F}$ such that $(P)$ holds for $\mathcal{F}\setminus\mathcal{E}$. Notice that if there is the extremal function for the $p$--modulus of $\mathcal{F}$ then any $p$--exceptional family is of measure zero. One can show \cite{bf} that if $f\in L^p(M)$, then $f\in L^1(L)$ for $p$--almost every $L\in\mathcal{F}$.

For any $\varphi\in L^p(M)$ let
\begin{equation*}
\hat{\varphi}(x)=\int_{L_x} \varphi\,d\mu_{L_x},
\end{equation*}
where $L_x\in\mathcal{F}$ is a leaf through $x\in M$. Function $\hat{\varphi}$ is defined $p$--almost everywhere and almost everywhere if there is the extremal function for the $p$--modulus of $\mathcal{F}$. We will often write $(f)^{\widehat{\,}}$ instead of $\hat{f}$.

Let us recall Badger's result \cite{mb} on the necessary and sufficient condition of existence of extremal function, which in our case takes the following form.
\begin{thm}\label{thm:badger}
There is the extremal function for the $p$--modulus of $\mathcal{F}$ if and only if
\begin{enumerate}
\item $\hat{f_0}$=1 (almost everywhere),
\item for any $\varphi\in L^p(M)$ such that $\hat{\varphi}\geq 0$ we have $\int_M f_0^{p-1}\varphi\,d\bar{\mu}\geq 0$.
\end{enumerate}
\end{thm}

Let
\begin{equation*}
L^p(\mathcal{F})=\{\varphi\in L^p(M)\mid {\rm esssup}|\hat{\varphi}|<\infty\}.
\end{equation*}
We have the following generalization of an integral formula obtained by the first author \cite{mc1}. This can be seen as a generalization of the Fubini theorem for a submersion, so called, co--area formula.
\begin{prop}
Assume there is extremal function $f_0$ for the $p$--modulus of $\mathcal{F}$. Then
\begin{equation}\label{eq:firstintegralfor}
\int_M f_0^{p-1}\varphi\hat{\psi}\,d\bar{\mu}=\int_M f_0^{p-1}\hat{\varphi}\psi\,d\bar{\mu}
=\int_M f_0^p \hat{\varphi}\hat{\psi}\,d\bar{\mu}.
\end{equation}
for any $\varphi,\psi\in L^p(\mathcal{F})$. In particular,
\begin{equation}\label{eq:secondintegralfor}
\int_M f_0^{p-1}\varphi\,d\bar{\mu}=\int_M f_0^p\hat{\varphi}\,d\bar{\mu},\quad \varphi\in L^p(\mathcal{F}).
\end{equation}
\end{prop}
\begin{proof}
Let $\varphi\in L^p(M)$ and $\hat{\varphi}=0$. Then $\pm\varphi$ satisfies the condition $(2)$ of Badger's theorem \ref{thm:badger}. Therefore $\pm\int_M f_0^{p-1}\varphi\,d\bar{\mu}\geq 0$, which implies $\int_M f_0^{p-1}\varphi\,d\bar{\mu}=0$. Now let $\varphi\in L^p(\mathcal{F})$. Then $\varphi-f_0\hat{\varphi}\in L^p(\mathcal{F})$ and $(\varphi-f_0\hat{\varphi})^{\widehat{\,}}=0$. Thus by above
\begin{equation*}
\int_M f_0^{p-1}(\varphi-f_0\hat{\varphi})=0,
\end{equation*}
which proves \eqref{eq:secondvariation}. Replacing $\varphi$ by $\hat{\varphi}\psi$ and then by $\varphi\hat{\psi}$ for $\varphi,\psi\in L^p(\mathcal{F})$, we get \eqref{eq:firstintegralfor}.
\end{proof}

Now we pass to the variation of $p$--modulus. Let $X$ be the compactly supported vector field on $M$, $\varphi_t$ the flow of $X$. Let $\mathcal{F}_t=\varphi_t(\mathcal{F})$. We say that $X$ is {\rm admissible} for the $p$--modulus of $\mathcal{F}$ if there is an open neighborhood $I=(-\varepsilon,\varepsilon)\subset\mathbb{R}$ such that
\begin{enumerate}
\item[A1)] there is the extremal function $f_t$ for the $p$--modulus of $\mathcal{F}_t$ for $t\in I$,
\item[A2)] the function $\alpha(x,t)=(f_t\circ\varphi_t)(x)$ is $C^2$--smooth with respect to variable $t\in I$, 
\item[A3)] the functions $\alpha,\frac{\partial \alpha}{\partial t}$ and $\frac{\partial^2\alpha}{\partial t^2}$, $t\in I$, are dominated by functions in $L^p(\mathcal{F})$ (independent of $t$).
\end{enumerate}

Moreover, we say that $\mathcal{F}$ is $p$--{\it admissible} if any compactly supported vector field on $M$ is admissible for the $p$--modulus of $\mathcal{F}$. The following theorem states the existence of $p$--admissible foliations. The proof of below result can be found in \cite{mc1} without the assumption on the second derivative of $\alpha$ in the definition of admissibility, but the proof in this case follows the same lines.  

\begin{thm}
Assume $\mathcal{F}$ is given by the level sets of a submersion $\Phi:M\to N$ such that $C_1<J\Phi<C_2$ for some positive constants $C_1,C_2$ and there exists extremal function for the $p$--modulus of $\mathcal{F}$. Then $\mathcal{F}$ is $p$--admissible.
\end{thm}

Now we can state the formula for the first variation of the $p$--modulus \cite{mc1}.

\begin{thm}
Let $\mathcal{F}$ be foliation on $M$ and $X$ compactly supported vector field admissible for the $p$--modulus of $\mathcal{F}$. Assume that the extremal function $f_0$ for the $p$--modulus of $\mathcal{F}$ is $C^1$--smooth. Then
\begin{equation}\label{eq:firstvariation}
\frac{d}{dt}{\rm mod}_p(\mathcal{F}_t)_{t=0}=-p\int_M f_0^{p-1}(Xf_0+f_0{\rm div}_{\mathcal{F}}(X))\,d\bar{\mu},
\end{equation}
where ${\rm div}_{\mathcal{F}}X$ denoted the leafwise divergence of $X$.
\end{thm}

We say that a $p$--admissible foliation $\mathcal{F}$ is a {\it critical point} of $p$--modulus functional if the first variation \eqref{eq:firstvariation} vanishes for any compactly supported vector field $X$. Since \eqref{eq:firstvariation} can be written in the form
\begin{equation*}
\frac{d}{dt}{\rm mod}_p(\mathcal{F}_t)_{t=0}=p\int_M g\left( \left(\frac{1}{q}\nabla^{\top}\log f_0^p-H_{\mathcal{F}^{\bot}}\right)-\left(\frac{1}{p}\nabla^{\bot}\log f_0^p-H_{\mathcal{F}}\right),X \right)\,d\bar{\mu},
\end{equation*}
where $H_{\mathcal{F}}$ and $H_{\mathcal{F}^{\bot}}$ denote the mean curvature vector of foliation $\mathcal{F}$ and distribution orthogonal $\mathcal{F}^{\bot}$, respectively, we get that $\mathcal{F}$ is a critical point of the $p$--modulus functional if and only if 
\begin{equation}\label{eq:f0criticalmain}
\nabla \log f_0^p=qH_{\mathcal{F}^{\bot}}+pH_{\mathcal{F}}.
\end{equation} 
Notice that the condition $\nabla^{\top} \log f_0^p=q H_{\mathcal{F}^{\bot}}$ holds trivially since the variation induced by a vector field $X$ tangent to $\mathcal{F}$ leaves $\mathcal{F}$ invariant. Therefore, the condition for the critical point reduces to the following
\begin{equation}\label{eq:criticalpoint}
\nabla^{\bot}\log f_0^p=pH_{\mathcal{F}}.
\end{equation}

\section{Second variation of $p$--modulus}

In this section we derive the formula for the second variation of the $p$--modulus. First we need some technical results.

Let $\mathcal{F}$ be a foliation on a Riemannian manifold $(M,g)$, $X$ a vector field on $M$ with the flow $\varphi_t$. Let $\mathcal{F}_t=\varphi_t(\mathcal{F})$ and $L_t=\varphi_t(L)$ for $L\in\mathcal{F}$. Denote by $J^r\varphi_t$ the Jacobian of the map $\varphi_t:L_r\to L_{r+t}$.

\begin{lem}\label{lem:tangentjacobian}
The following relations hold
\begin{align}
\frac{d}{dt}\left(J^0\varphi_t\right)_{t=r} &=J^0\varphi_r\cdot{\rm div}_{\mathcal{F}_r}X\circ\varphi_r,\label{eq:tech2}\\
\frac{d}{dt}\left(J^r\varphi_t\right)_{t=0} &={\rm div}_{\mathcal{F}_r}X,\label{eq:tech1}\\
\frac{d^2}{dt^2}\left(J^0\varphi_t\right)_{t=0} &=({\rm div}_{\mathcal{F}}X)^2-{\rm Ric}_{\mathcal{F}}(X)+{\rm div}_{\mathcal{F}}(\nabla_XX)+\sum_i|\nabla_{e_i}X|^2 \label{eq:tech3}\\
&-\sum_{i,j}g(\nabla_{e_i}X,e_j)^2
-\sum_{i,j}g(\nabla_{e_i}X,e_j)g(\nabla_{e_j}X,e_i),\notag
\end{align}
where $(e_i)$ is an orthonormal basis of $T\mathcal{F}$.
\end{lem}
\begin{proof}
First let us recall the formulas for the differential of the determinant. Namely, for $1$--parameter family of matrices $Y_t$ such that $Y_0$ is an identity matrix, we have
\begin{align*}
\frac{d}{dt}(\det Y_t)_t &=\det (Y_t){\rm tr}\left(Y_t^{-1}\frac{d}{dt}Y_t\right),\\
\frac{d^2}{dt^2}(\det Y_t)_{t=0} &={\rm tr}\left(\frac{d^2}{dt^2}Y_t\right)_{t=0}+\left({\rm tr}\left(\frac{d}{dt}Y_t\right)_{t=0}\right)^2-{\rm tr}\left(\left(\frac{d}{dt}Y_t\right)^2_{t=0}\right).
\end{align*}
Let $Y_t$ be of the form
\begin{equation*}
(Y_t)_{ij}=g(\varphi_{t\ast}e_i,\varphi_{t\ast}e_j).
\end{equation*}
We will derive the first and second derivative of $Y_t$. Let $\phi(x,t)=\varphi_t(x)$ and consider the pull--back bundle $\phi^{-1}TM$ over $M\times\mathbb{R}$ with the fibers $(\phi^{-1}TM)_{(x,t)}=T_{\varphi_t(x)}M$. There is unique connection $\nabla^{\phi}$ in this bundle satisfying
\begin{equation*}
\nabla^{\phi}_W (Y\circ\phi)=\nabla_{\phi_{\ast}W}Y,\quad W\in T(M\times\mathbb{R}),\quad Y\in\Gamma(TM).
\end{equation*}
Moreover, one can show that \cite{bw}
\begin{equation*}
\nabla^{\phi}_W\phi_{\ast}Z-\nabla^{\phi}_Z\phi_{\ast}W=\phi_{\ast}[W,Z],
\quad W,Z\in\Gamma(T(M\times\mathbb{R})).
\end{equation*}
Since $[e_i,\frac{d}{dt}]=0$ we get
\begin{align*}
\frac{d}{dt}(Y_t)_{ij} &=\frac{d}{dt}g(\varphi_{t\ast}e_i,\varphi_{t\ast}e_j)\\
&=g(\nabla^{\phi}_{\frac{d}{dt}}\varphi_{t\ast}e_i,\varphi_{t\ast}e_j)
+g(\nabla^{\phi}_{\frac{d}{dt}}\varphi_{t\ast}e_j,\varphi_{t\ast}e_i)\\
&=g(\nabla^{\phi}_{e_i}\phi_{\ast}\frac{d}{dt},\varphi_{t\ast}e_j)
+g(\nabla^{\phi}_{e_j}\phi_{\ast}\frac{d}{dt},\varphi_{t\ast}e_i).
\end{align*}
By the fact that $\phi_{\ast}\frac{d}{dt}=X\circ\phi$ we have
\begin{equation*}
\frac{d}{dt}(Y_t)_{ij}=g(\nabla_{\varphi_{t\ast}e_i}X,\varphi_{t\ast}e_j)
+g(\nabla_{\varphi_{t\ast}e_j}X,\varphi_{t\ast}e_i).
\end{equation*}
In particular 
\begin{equation*}
{\rm tr}\left(Y_t^{-1}\frac{d}{dt}Y_t\right)=
2\sum_{i,j}(g(\varphi_{t\ast}e_i,\varphi_{t\ast}e_j))^{-1}_{ij}g(\nabla_{\varphi_{t\ast}e_i}X,\varphi_{t\ast}e_j)
=2{\rm div}_{\mathcal{F}_t}X\circ\varphi_t
\end{equation*}
and hence
\begin{equation*}
{\rm tr}\left(\frac{d}{dt}Y_t\right)_{t=0}=2{\rm div}_{\mathcal{F}}X.
\end{equation*}
Therefore
\begin{align*}
\frac{d}{dt}(J^0\varphi_t)_t &=\frac{d}{dt}\sqrt{\det Y_t}=\frac{1}{2\sqrt{\det Y_t}}\frac{d}{dt}(\det Y_t)
=\frac{1}{2}\sqrt{\det Y_t}{\rm tr}(Y_t^{-1}\frac{d}{dt} Y_t)\\
&=J^0\varphi_t\cdot {\rm div}_{\mathcal{F}_t}X\circ\varphi_t.
\end{align*}
By the correspondence
\begin{equation*}
J^0\varphi_{r+t}=J^0\varphi_r\cdot (J^r\varphi_t\circ\varphi_r)
\end{equation*}
we obtain
\begin{align*}
\frac{d}{dt}(J^r\varphi_t)_{t=0}\circ\varphi_r &=\frac{\frac{d}{dt}(J^0\varphi_{r+t})_{t=0}}{J^0\varphi_r}
=\frac{\frac{d}{dt}(J^0\varphi_t)_{t=r}}{J^0\varphi_r}={\rm div}_{\mathcal{F}_r}X\circ\varphi_r.
\end{align*}
Furthermore, up to a symmetrization of indices $i,j$,
\begin{align*}
\left(\frac{d^2}{dt^2}(Y_t)_{ij}\right)_{t=0} &=\frac{d}{dt}g(\nabla_{\varphi_{t\ast}e_i}X,\varphi_{t\ast}e_j)_{t=0}\\
&=
g((\nabla^{\phi}_{\frac{d}{dt}}\nabla_{\varphi_{t\ast}e_i}X)_{t=0},e_j)
+g(\nabla_{\varphi_{t\ast}e_i}X,\nabla^{\phi}_{\frac{d}{dt}}\varphi_{t\ast}e_j)_{t=0}\\
&=g(\nabla^{\phi}_{\frac{d}{dt}}\nabla^{\phi}_{e_i}(X\circ\phi),e_j)+g(\nabla_{e_i}X,\nabla_{e_j}X)\\
&=g(R^{\phi}(\frac{d}{dt},e_i)(X\circ\varphi),e_j)+g(\nabla^{\phi}_{e_i}\nabla^{\phi}_{\frac{d}{dt}}(X\circ\phi),e_j)
+g(\nabla_{e_i}X,\nabla_{e_j}X)\\
&=g(R(X,e_i)X,e_j)+g(\nabla_{e_i}\nabla_XX,e_j)+g(\nabla_{e_i}X,\nabla_{e_j}X),
\end{align*}
where $R^{\phi}$ is the curvature tensor of the connection $\nabla^{\phi}$ and $R$ is the curvature tensor on $M$. Hence
\begin{align*}
\sum_i\left(\frac{d^2}{dt^2}(Y_t)_{ii}\right)_{t=0} &=2\left(-{\rm Ric}(X)+{\rm div}_{\mathcal{F}}(\nabla_XX)+\sum_i|\nabla_{e_i}X|^2\right). 
\end{align*}
Moreover,
\begin{align*}
{\rm tr}\left(\frac{d}{dt}Y_t\right)^2_{t=0} &=\sum_{i,j} \left(\left(\frac{d}{dt}Y_t\right)_{t=0}\right)_{ij}^2\\
&=\sum_{i,j}(g(\nabla_{e_i}X,e_j)+g(\nabla_{e_j}X,e_i))^2\\
&=2\sum_{i,j}(g(\nabla_{e_i}X,e_j)^2+g(\nabla_{e_i}X,e_j)g(\nabla_{e_j}X,e_i)).
\end{align*}
Since
\begin{align*}
\frac{d^2}{dt^2}(J^0\varphi_t)_{t=0} &=\frac{d^2}{dt^2}\left(\sqrt{\det Y_t}\right)_{t=0}\\
&=\frac{1}{2}\frac{d}{dt}\left( (\det Y_t)^{-\frac{1}{2}}\frac{d}{dt}(\det Y_t)\right)_{t=0}\\
&=\frac{1}{2}\left( -\frac{1}{2}\left(\frac{d}{dt}(\det Y_t)_{t=0}\right)^2+\frac{d^2}{dt^2}(\det Y_t)_{t=0} \right),
\end{align*}
by above we get
\begin{align*}
2\frac{d^2}{dt^2}(J^0\varphi_t)_{t=0} &=-\frac{1}{2}(2{\rm div}_{\mathcal{F}}X)^2-2\left(-{\rm Ric}(X)+{\rm div}_{\mathcal{F}}(\nabla_XX)+\sum_i|\nabla_{e_i}X|^2\right)\\
&+(2{\rm div}_{\mathcal{F}}X)^2-2\sum_{i,j}g(\nabla_{e_i}X,e_j)^2
-2\sum_{i,j}g(\nabla_{e_i}X,e_j)g(\nabla_{e_j}X,e_i).
\end{align*}
Finally
\begin{align*}
\frac{d^2}{dt^2}(J^0\varphi_t)_{t=0} &=({\rm div}_{\mathcal{F}}X)^2-{\rm Ric}_{\mathcal{F}}(X)+{\rm div}_{\mathcal{F}}(\nabla_XX)+\sum_i|\nabla_{e_i}X|^2\\
&-\sum_{i,j}g(\nabla_{e_i}X,e_j)^2
-\sum_{i,j}g(\nabla_{e_i}X,e_j)g(\nabla_{e_j}X,e_i).\qedhere
\end{align*}
\end{proof}

The same argument as in the proof of above lemma implies that the second derivative of the ''full'' Jacobian $J\varphi_t:M\to M$ equals
\begin{equation}\label{eq:fulljacobian}
\frac{d^2}{dt^2}(J\varphi_t)_{t=0}=(\rm div X)^2-{\rm Ric}(X)+{\rm div}(\nabla_XX)-{\rm tr}((\nabla X)^2).
\end{equation}

Assume now that $X$ is admissible for the $p$--modulus of $\mathcal{F}$. Recall the definition of the function $\alpha_t(x)=\alpha(t,x)$,
\begin{equation*}
\alpha_t(x)=(f_t\circ \varphi_t)(x)
\end{equation*}
where $f_t$ is the extremal function for the $p$--modulus of $\mathcal{F}_t$ and $\varphi_t$ is the flow of $X$. By admissibility of $X$ the function $t\mapsto \alpha_t$ is twice differentiable. 

\begin{lem}\label{lem:hatdftdt}
The following relations hold
\begin{align*}
&\left( \frac{d\alpha_t}{dt} \right)_{t=0}^{\widehat{\,}}=-\left(f_0{\rm div}_{\mathcal{F}}X\right)^{\widehat{\,}},\\
&\left(\left( \frac{d^2\alpha_t}{dt^2}\right)_{t=0}+2\left(\frac{d\alpha_t}{dt}\right)_{t=0}{\rm div}_{\mathcal{F}}X\right)^{\widehat{\,}}=-\left(f_0\frac{d^2}{dt^2}J^{\top}\varphi_t\right)^{\widehat{\,}}.
\end{align*}
\end{lem}
\begin{proof}
It suffices to differentiate the equation
\begin{equation*}
1=\widehat{f_t}^t\circ\varphi_t=\widehat{(f_t\circ\varphi_t)J^0\varphi_t}
\end{equation*}
and use Lemma \ref{lem:tangentjacobian}.
\end{proof}

\begin{lem}\label{lem:diffintfor}
Let $p\geq 2$ and let $\psi_t$ be $1$--parameter family of functions on $M$ such that for $t$ in some interval $(-\varepsilon,\varepsilon)$ we have
\begin{enumerate}
\item $\psi_t\in L^p(\mathcal{F}_t)$,
\item $\psi_t\circ\varphi_t$ is differentiable with respect to $t$,
\item $\psi_t\circ\varphi_t$, $\frac{d}{dt}(\psi_t\circ\varphi_t)$ are dominated by some functions in $L^p(\mathcal{F})$ (independent of $t\in(-\varepsilon,\varepsilon)$).
\end{enumerate}
Then the following formula holds
\begin{multline*}
(p-1)\int_M f_0^{p-2}\left(\frac{d\alpha_t}{dt}\right)_{t=0}\psi\,d\bar{\mu}\\
=(1-p)\int_M f_0^p\hat{\psi}{\rm div}_{\mathcal{F}}X\,d\bar{\mu}
+\int_M f_0^p\hat{\psi}{\rm div}_{\mathcal{F}^{\bot}}X\,d\bar{\mu}
-\int_M f_0^{p-1}\psi{\rm div}_{\mathcal{F}^{\bot}}X\,d\bar{\mu},
\end{multline*} 
where $\psi=\psi_0$.
\end{lem}
\begin{proof}
The integral formula \eqref{eq:secondintegralfor} for $\mathcal{F}_t$ takes the form
\begin{equation*}
\int_M f_t^{p-1}\psi_t\,d\bar{\mu}=\int_M f_t^p\hat{\psi_t}^t\,d\bar{\mu},
\end{equation*}
or equivalently
\begin{equation*}
\int_M \alpha_t^{p-1}(\psi_t\circ\varphi_t)J\varphi_t\,d\bar{\mu}=\int_M \alpha_t^p(\hat{\psi_t}^t\circ\varphi_t)J\varphi_t\,d\bar{\mu}.
\end{equation*}
By the assumptions and Lebesgue dominated convergence theorem we can differentiate under the integral sign. Moreover, by Lemma \ref{lem:tangentjacobian}
\begin{equation*}
\frac{d}{dt}(\hat{\psi_t}^t\circ\varphi_t)=\frac{d}{dt}\int_L (\psi_t\circ\varphi_t)J^0\varphi_t\,d\mu_L=\left(\frac{d}{dt}(\psi_t\circ\varphi_t)+\psi {\rm div}_{\mathcal{F}}X\right)^{\widehat{\,}}.
\end{equation*}
Thus (all derivatives with respect to $t$ are taken at $t=0$)
\begin{multline*}
\int_M \left((p-1)f_0^{p-2}\frac{d\alpha_t}{dt}\psi+f_0^{p-1}\frac{d}{dt}(\psi_t\circ\varphi_t)+f_0^{p-1}\psi{\rm div}X\right)\,d\bar{\mu}\\
=\int_M \left(pf_0^{p-1}\frac{d\alpha_t}{dt}\hat{\psi}+f_0^p\left(\frac{d}{dt}(\psi_t\circ\varphi_t)+\psi {\rm div}_{\mathcal{F}}X\right)^{\widehat{\,}}+f_0^p\hat{\psi}{\rm div}X\right)\,d\bar{\mu}.
\end{multline*}
By \eqref{eq:secondintegralfor} above equality simplifies to
\begin{multline*}
(p-1)\int_M f_0^{p-2}\frac{d\alpha_t}{dt}\psi\,d\bar{\mu}\\
=p\int_M f_0^{p-1}\frac{d\alpha_t}{dt}\hat{\psi}\,d\bar{\mu}
+\int_M f_0^p\hat{\psi}{\rm div}X\,d\bar{\mu}
-\int_M f_0^{p-1}\psi{\rm div}_{\mathcal{F}^{\bot}}X\,d\bar{\mu}. 
\end{multline*}
Finally, by \eqref{eq:firstintegralfor} and Lemma \ref{lem:hatdftdt} we get the desired equality.
\end{proof}

Now we can state and prove the formula for the second variation.  

\begin{thm}\label{thm:secondvariation}
Let $p\geq 2$. Let $\mathcal{F}$ be a foliation on $M$ and $X$ compactly supported vector field admissible for the $p$--modulus of $\mathcal{F}$. Then
\begin{equation}\label{eq:secondvariation}
\begin{split}
\frac{d^2}{dt^2}{\rm mod}_p(\mathcal{F}_t)_{t=0} &=\int_M f_0^p\frac{d^2}{dt^2}\left( J\varphi_t-pJ^0\varphi_t \right)_{t=0}\,d\bar{\mu}-q\int_M f_0^p({\rm div}_{\mathcal{F}^{\bot}}X)^2\,d\bar{\mu}\\
&+p\int_M f_0^p((f_0(\sqrt{p-1}{\rm div}_{\mathcal{F}}X-\sqrt{q-1}{\rm div}_{\mathcal{F}^{\bot}}X))^{\widehat{\,}})^2\,d\bar{\mu},
\end{split}
\end{equation}
where $\frac{d^2}{dt^2}(J\varphi_t)_{t=0}$ and $\frac{d^2}{dt^2}(J^0\varphi_t)_{t=0}$ are given by \eqref{eq:tech3} and \eqref{eq:fulljacobian}, respectively.
\end{thm}
\begin{proof}
By $p$--admissibility of $X$
\begin{align*}
\left|\frac{d^2}{dt^2}\left((f_t\circ\varphi_t)^pJ\varphi_t\right)\right| &\leq
\left|p(p-1)(f_t\circ\varphi_t)^{p-2}\left(\frac{d}{dt}(f_t\circ\varphi_t)\right)^2\right||J\varphi_t|\\
&+\left|p(f_t\circ\varphi_t)^{p-1}\frac{d^2}{dt^2}(f_t\circ\varphi_t)\right||J\varphi_t|\\
&+2\left|p(f_t\circ\varphi_t)^{p-1}\frac{d}{dt}(f_t\circ\varphi_t)\right|\left|\frac{d}{dt}J\varphi_t\right|
+|(f_t\circ\varphi_t)^p|\left|\frac{d^2}{dt^2}J\varphi_t\right|\\
&\leq C(h_1^{p-2}h_2^2+h_1^{p-1}h_3+h_1^{p-1}h_2+h_1^p),
\end{align*}
for some constant $C>0$ and $p$--integrable functions $h_1$ and $h_2$. It follows by H\"older inequality that the function $h_1^{p-2}h_2^2+h_1^{p-1}h_3+h_1^{p-1}h_2+h_1^p$ is integrable. Thus by Lebesgue dominated convergence theorem (we compute derivatives with respect to $t$ at $t=0$)
\begin{align*}
\frac{d^2}{dt^2}{\rm mod}_p(\mathcal{F}_t) &=\frac{d^2}{dt^2}\int_Mf_t^p\,d\bar{\mu}\\
&=\frac{d^2}{dt^2}\int_M \alpha_t^pJ\varphi_t\,d\bar{\mu}\\
&=\int_M \frac{d^2}{dt^2}\left( \alpha_t^pJ\varphi_t)\right)\,d\bar{\mu}\\
&=\int_M p(p-1)f_0^{p-2}\left(\frac{d\alpha_t}{dt}\right)^2 \,d\bar{\mu}
+\int_M p f_0^{p-1}\frac{d^2\alpha_t}{dt^2}\,d\bar{\mu}\\
&+\int_M \left(2p f_0^{p-1}\frac{d\alpha_t}{dt}{\rm div}X
+f_0^p\frac{d^2}{dt^2}J\varphi_t\right)\,d\bar{\mu}.
\end{align*}
By Lemma \ref{lem:hatdftdt} and integral formula \eqref{eq:secondintegralfor} we get
\begin{align*}
\frac{d^2}{dt^2}{\rm mod}_p(\mathcal{F}_t) &=\int_M f_0^p\frac{d^2}{dt^2}(J\varphi_t-pJ^0\varphi_t)\,d\bar{\mu}
+p(p-1)\int_M f_0^{p-2}\left(\frac{d\alpha_t}{dt}\right)^2\,d\bar{\mu}\\
&+2p\int_M f_0^{p-1}\frac{d\alpha_t}{dt}{\rm div}_{\mathcal{F}^{\bot}}X\,d\bar{\mu}
\end{align*}
Now we will use Lemma \ref{lem:diffintfor} for $\psi_t=\left(\frac{d\alpha_t}{dt}\right)_{t=0}\circ\varphi_{-t}$ and $\psi_t=f_t{\rm div}_{\mathcal{F}^{\bot}}X$. By $p$--admissibility of $\mathcal{F}$ these functions satisfy the assumptions of Lemma \ref{lem:diffintfor}. Hence, after some computations (with the use of integral formula \eqref{eq:secondintegralfor}), we get
\begin{align*}
\frac{d^2}{dt^2}{\rm mod}_p(\mathcal{F}_t) &=\int_M f_0^p\frac{d^2}{dt^2}(J\varphi_t-pJ^0\varphi_t)\,d\bar{\mu}
-q\int_M f_0^p({\rm div}_{\mathcal{F}^{\bot}}X)^2\,d\bar{\mu}\\
&+p(p-1)\int_M f_0^p\widehat{(f_0{\rm div}_{\mathcal{F}}X)}^2\,d\bar{\mu}
-2p\int_M f_0^p\widehat{(f_0{\rm div}_{\mathcal{F}}X)}\widehat{(f_0{\rm div}_{\mathcal{F}^{\bot}}X)}\,d\bar{\mu}\\
&+q\int_M f_0^p \widehat{(f_0{\rm div}_{\mathcal{F}^{\bot}}X)}^2\,d\bar{\mu}.
\end{align*}
It suffices to notice that last three components sum up to
\begin{equation*}
p\int_M f_0^p((f_0\sqrt{p-1}{\rm div}_{\mathcal{F}}X-\sqrt{q-1}{\rm div}_{\mathcal{F}^{\bot}}X)^{\widehat{\,}})^2\,d\bar{\mu}.
\end{equation*}
\end{proof}

By the monotonicity of the $p$--modulus we get that the $p$--modulus of a $k$--dimensional foliation $\mathcal{F}$ is less or equal to the $p$--modulus of the family of all $k$--dimensional submanifolds of given Riemannian manifold $M$. Moreover, if $M$ is a doubly connected set in the Euclidean space $\mathbb{R}^n$, then the $p$--modulus of all hypersurfaces separating boundary components equals the $q$--capacity of a considered set and the $p$--modulus of a foliation $\mathcal{F}_u$ given by the level sets of extremal function for the $q$--capacity \cite{wz}. Therefore, $\mathcal{F}_u$ maximizes the $p$--modulus among all foliations separating boundary components. This leads to the following definition.

 We say that a $p$--admissible foliation $\mathcal{F}$ is $p$--{\it stable} if it is a critical point of the $p$--modulus functional and if the second variation is nonpositive for any compactly supported vector field $X$ orthogonal to $\mathcal{F}$. We only require $X$ to be orthogonal to $\mathcal{F}$ since the variation by the vector field tangent to foliation $\mathcal{F}$ is constant. 

Let us now consider the codimension one case. Assume moreover that there exists unit normal vector field $N$ to $\mathcal{F}$. Then any compactly supported vector field orthogonal to $\mathcal{F}$ is of the form $X=fN$, where $f\in C^{\infty}_0(M)$. For simplicity let
\begin{align*}
A &=\int_M f_0^p\frac{d^2}{dt^2}\left( J\varphi_t-pJ^0\varphi_t \right)_{t=0}\,d\bar{\mu}-q\int_M f_0^p({\rm div}_{\mathcal{F}^{\bot}}X)^2\,d\bar{\mu}\\
B &=p\int_M f_0^p((f_0(\sqrt{p-1}{\rm div}_{\mathcal{F}}X-\sqrt{q-1}{\rm div}_{\mathcal{F}^{\bot}}X))^{\widehat{\,}})^2\,d\bar{\mu}.
\end{align*}

If $\mathcal{F}$ is a critical point of the $p$--modulus functional and the extremal function $f_0$ for the $p$--modulus of $\mathcal{F}$ is $C^1$--smooth, then rewriting \eqref{eq:f0criticalmain} we have for any vector field $Y$
\begin{equation}\label{eq:criticalpointcodim1}
Yf_0^p=f_0^pg(qH_{\mathcal{F}^{\bot}}+pH_{\mathcal{F}},Y).
\end{equation}
Notice that $H_{\mathcal{F}}=h_{\mathcal{F}}N$, where $h_{\mathcal{F}}$ has real values and $H_{\mathcal{F}^{\bot}}=\nabla_NN$. Moreover the Ricci curvature of $\mathcal{F}$ in the direction of $N$ equals the Ricci curvature of $M$ in the direction of $N$. By a simple computations we get
\begin{align*}
{\rm div}(\nabla_XX) &=f^2{\rm div}(\nabla_NN)+(\nabla_NN)f^2+{\rm div}(f(Nf)N),\\
{\rm div}_{\mathcal{F}}(\nabla_XX) &=f^2{\rm div}(\nabla_NN)+f^2|\nabla_NN|^2+(\nabla_NN)f^2-f(Nf)h_{\mathcal{F}},\\
\sum_i g(\nabla_{e_i}X,N)^2 &=|\nabla^{\top} f|^2,\\
\sum_{i,j}g(\nabla_{e_i}X,e_j)g(\nabla_{e_j}X,e_i) &=f^2|\Pi|^2,\\
{\rm tr}((\nabla X)^2) &=f|\Pi|^2+(\nabla_NN)f^2+(Nf)^2,
\end{align*}
where $\Pi$ is the second fundamental form of $\mathcal{F}$, $\Pi(X,Y)=(\nabla_XY)^{\bot}$, $X,Y\in TM$. Thus 
\begin{align*}
A &=\int_M f_0^p\Big( (1-p)f^2h_{\mathcal{F}}^2+(p-1)f^2{\rm Ric}(N)-p|\nabla^{\top}f|^2-q(Nf)^2\\
&+(p-1)f^2|\Pi|^2-pf^2|\nabla_NN|^2+(1-p)f^2{\rm div}(\nabla_NN)-p(\nabla_NN)f^2\\
&-2f(Nf)h_{\mathcal{F}}+{\rm div}(f(Nf)N)+pf(Nf)h_{\mathcal{F}}\Big)\,d\bar{\mu}.
\end{align*}
The last two components sum up to ${\rm div}(f_0^pf(Nf)N)$. Computing the divergence of the vector fields $f_0^pf^2\nabla_NN$ and $f_0^pf^2h_{\mathcal{F}}N$ and using \eqref{eq:criticalpointcodim1} for $Y=\nabla_NN$ we obtain
\begin{align*}
A &=\int_M f_0^p\Big(-p|\nabla^{\top}f|^2-q(Nf)^2+q|\nabla_NN|^2+(p-1)f^2|\Pi|^2\\
&+(p-1)f^2{\rm Ric}(N)+f^2{\rm div}(\nabla_NN)+f^2N(h_{\mathcal{F}})\Big)\,d\bar{\mu}.
\end{align*}
By the formula \cite{bko} 
\begin{equation}\label{eq:bkoformula}
{\rm div}(\nabla_NN)=-N(h_{\mathcal{F}})+|\Pi|^2+{\rm Ric}(N).
\end{equation}
we get
\begin{equation*}
A=\int_M f_0^p\left(pf^2{\rm Ric}(N)-p|\nabla^{\top} f|^2-q(Nf)^2
+pf^2|\Pi|^2+qf^2|\nabla_NN|^2 \right)\,d\bar{\mu}.
\end{equation*}
Moreover,
\begin{equation*}
B=p\int_M f_0^p((f_0(\sqrt{p-1}fh_{\mathcal{F}}+\sqrt{q-1}Nf))^{\widehat{\,}})^2\,d\bar{\mu}.
\end{equation*}

Put
\begin{equation*}
\nabla_{p,q}f=\sqrt{p}\nabla^{\top}f+\sqrt{q}\nabla^{\bot}f=\sqrt{p}\nabla^{\top}f+\sqrt{q} Nf.
\end{equation*}
Now we can state one of the main theorems.
\begin{thm}\label{thm:codim1}
Let $p\geq 2$. Assume $\mathcal{F}$ is $p$--admissible codimension one transversally orientable foliation on a Riemannian manifold $M$ and let $N$ be the unit normal vector field. Assume $\mathcal{F}$ is a critical point of the $p$--modulus functional and $f_0$ is $C^1$--smooth. Then $\mathcal{F}$ is $p$--stable if and only if 
\begin{multline}\label{eq:codim1}
\int_M f_0^p(-|\nabla_{p,q}f|^2+pf^2|\Pi|^2+qf^2|\nabla_NN|^2+pf^2{\rm Ric}(N))\,d\bar{\mu}\\
+p\int_M f_0^p((f_0(\sqrt{p-1}fh_{\mathcal{F}}+\sqrt{q-1}Nf))^{\widehat{\,}})^2\,d\bar{\mu}\leq 0.
\end{multline}
for any $f\in C^{\infty}_0(M)$.
\end{thm}

We will give the sufficient condition for the $p$--stability. Let $\mathcal{F}$ be a codimension one transversally orientable foliation on a Riemannian manifold $M$ and let $N$ be the unit normal vector field. Denote by $\alpha_0$ the following function
\begin{equation*}
\alpha_0=\frac{(f_0^2)^{\widehat{\,}}}{f_0},
\end{equation*}
where $f_0$ is the extremal function for the $p$--modulus of $\mathcal{F}$.
\begin{cor}\label{cor:codimonestable}
Let $p\geq 2$. Assume $\mathcal{F}$ is $p$--admissible and a critical point of the $p$--modulus functional and $f_0$ is $C^1$--smooth. If
\begin{multline}\label{eq:suffstable}
\int_M f_0^p\Big( -p|\nabla^{\top}f|^2-q(1-\alpha_0)(Nf)^2+pf^2(1-\alpha_0){\rm Ric}(N)+pf^2(1-\alpha_0)f^2|\Pi|^2\\
+qf^2|\nabla_NN|-p(N\alpha_0)f^2h_{\mathcal{F}}+p\alpha_0f^2{\rm div}(\nabla_NN)\Big)\,d\bar{\mu}\leq 0
\end{multline}
for any $f\in C^{\infty}_0(M)$, then $\mathcal{F}$ is $p$--stable.
\end{cor}
\begin{proof}
Denote the function $\sqrt{p-1}fh_{\mathcal{F}}+\sqrt{q-1}Nf$ by $\psi$. Then, by H\"older inequality and integral formula \eqref{eq:firstintegralfor}, we have
\begin{equation*}
\int_M f_0^p\widehat{f_0\psi}^2\,d\bar{\mu}\leq \int_M f_0^p \widehat{f_0^2}\widehat{\psi^2}\,d\bar{\mu}
=\int_M f_0^{p-1}\widehat{f_0^2}\psi^2\,d\bar{\mu}=\int_M f_0^p\alpha_0\psi^2\,d\bar{\mu}.
\end{equation*}
It suffices to use Theorem \ref{thm:codim1}, equality \eqref{eq:bkoformula} and compute the divergence of the vector field $f_0^p\alpha_0f^2h_{\mathcal{F}}N$.
\end{proof}

\begin{cor}
Let $\mathcal{F}$ be a foliation given by the level sets of the distance function (from the codimension one submanifold $L$ in $M$). Assume $f_0$ is $C^1$--smooth. Then $\mathcal{F}$ is $p$--stable for any $p\geq 2$.
\end{cor}
\begin{proof}
$\mathcal{F}$ is $p$--stable for any $p>1$ (see the next section). Thus $f_0$ is constant on the leaves of $\mathcal{F}$, hence $f_0=\frac{1}{\sigma}$, where $\sigma(x)$, $x\in M$, is the measure of the leaf through $x$. Thus
\begin{equation*}
\alpha_0=\frac{1}{f_0}\widehat{f_0^2}=\sigma \frac{1}{\sigma}=1.
\end{equation*}
Since $\mathcal{F}^{\bot}$ is totally geodesic, then condition \eqref{eq:suffstable} takes the form
\begin{equation*}
-p\int_M f_0^p|\nabla^{\top}f|^2\leq 0,
\end{equation*}
which clearly holds for any $f\in C^{\infty}_0(M)$.
\end{proof}

At the end of this section we will describe the second variation in terms of extremal function $f_0$.
\begin{prop}\label{prop:sigma2f0}
Assume $\mathcal{F}$ is $p$--admissible codimension one transversally orientable and $N$ is the unit normal vector field. Assume moreover that the extremal function $f_0$ for the $p$--modulus of $\mathcal{F}$ is $C^2$--smooth. If $\mathcal{F}$ is a critical point of the $p$--modulus functional, then the following formula holds
\begin{equation}\label{eq:sigma2f0}
p|\Pi|^2+q|\nabla_NN|^2+p{\rm Ric}(N)=(p-1)\Delta_{\mathcal{F}} k_0+(2-p)|\nabla^{\top}k_0|+N^2k_0,
\end{equation}
where $k_0=\log f_0^p$. In particular, $\mathcal{F}$ is $p$--stable if and only if
\begin{multline}\label{eq:codim1f_0}
\int_M e^{k_0}\left(-|\nabla_{p,q}f|^2+f^2\left( (p-1)\Delta_{\mathcal{F}} k_0+(2-p)|\nabla^{\top}k_0|+N^2k_0 \right)\right)\,d\bar{\mu}\\
+p\int_Me^{k_0}\left(\left( e^{\frac{1}{p}k_0}\left(\frac{\sqrt{p-1}}{p}(Nk_0)f+\sqrt{q-1}Nf \right)\right)^{\widehat{\quad}}\right)^2\,d\bar{\mu}\leq 0
\end{multline}
for any $f\in C^{\infty}_0(M)$.
\end{prop}
\begin{proof}
Follows directly by \eqref{eq:f0criticalmain} and \eqref{eq:bkoformula}.
\end{proof}

\section{Level sets of $q$--harmonic functions}

In this section we study the $p$--modulus of foliations given by the level sets of functions with non--vanishing gradient. We show the relation between $q$--harmonicity of functions defining foliations and $p$--modulus of these foliations.

First, recall that a $C^2$--smooth function $u:M\to \mathbb{R}$ is $q$--harmonic, if it satisfies the following differential equation
\begin{equation*}
\Delta_q u= {\rm div}(|\nabla u|^{q-2}\nabla u)=0.
\end{equation*}

Let $u:M\to \mathbb{R}$ be $C^2$--smooth and assume $|\nabla u|>0$. Denote by $\mathcal{F}_u$ the foliation given by the level sets of $u$. Let $\nu^q_u:\mathbb{R}\to\mathbb{R}$ be defined as follows
\begin{equation*}
\nu^q_u(t)=\int_{u^{-1}(t)}|\nabla u|^{q-1}\,d\mu_{u^{-1}(t)}=\left(\widehat{|\nabla u|^{q-1}}\right)\circ u^{-1}
\end{equation*}
Notice that by \eqref{eq:modsubmersion} we have
\begin{equation*}
{\rm mod}_p(\mathcal{F}_u)=\int_{u(M)} \nu^q_u(t)^{1-p}\,dt.
\end{equation*}
Consider the vector field $Y=\frac{\nabla u}{|\nabla u|^2}$. Then
\begin{equation*}
u_{\ast} Y=g(\nabla u,Y)\left(\frac{d}{dt}\right)\circ u=\frac{d}{dt}\circ u.
\end{equation*}
Fix $s\in u(M)$. If $\psi_t$ is a local flow of $Y$ in the neighborhood of $u^{-1}(s)$ then $u\circ\psi_t=\lambda_t\circ u$, where $\lambda_t(\theta)=\theta+t$. Hence $\psi_t$ maps $u^{-1}(s)$ to $u^{-1}(s+t)$. 

\begin{prop}\label{prop:diffnu}
Assume $u:M\to\mathbb{R}$ is $C^2$--smooth. If the $1$--parameter family of functions $\left((\frac{\Delta_qu}{|\nabla u|}\circ\psi_t\right)J^0\psi_t$, $t\in(-\varepsilon,\varepsilon)$, is dominated by a function (independent of $t$) integrable on the level sets $u^{-1}(s+t)$, $t\in(\varepsilon,\varepsilon)$, then $\nu^q_u$ is differentiable at $s$ and
\begin{equation}\label{eq:diffnu}
\left(\frac{d}{dt}\nu^q_u\right)(s)=\widehat{\left(\frac{\Delta_q u}{|\nabla u|}\right)}\circ u^{-1}(s).
\end{equation} 
\end{prop}
\begin{proof}
Since
\begin{equation*}
\nu^q_u(s+t)=\int_{u^{-1}(s)}(|\nabla u|^{q-1}\circ\psi_t)J^0\psi_t\,d\mu_{u^{-1}(s)}
\end{equation*}
we have by Lagrange mean value theorem
\begin{equation*}
\lim_{t\to 0}\frac{1}{t}\left(\nu^q_u(s+t)-\nu^q_u(s)\right)=\lim_{t\to 0}\int_{u^{-1}(s)}\frac{d}{dt}\left( |\nabla u|^{q-1}\circ\psi_t\cdot J^0\psi_t \right)_{\theta_t},
\end{equation*}
where $\theta_t\in (-t,t)$. By easy computations we get
\begin{align*}
\frac{d}{dt}\left( (|\nabla u|^{q-1}\circ\psi_t) J^0\psi_t \right) &=\left(Y(|\nabla u|^{q-1})+|\nabla u|^{q-1}{\rm div}_{\mathcal{F}_u}Y\right)\circ\psi_t\cdot J^0\psi_t\\
&=\left(\frac{1}{|\nabla u|}Y|\nabla u|^q+|\nabla u|^qY\left(\frac{1}{|\nabla u|}\right)\right.\\
&\left.+|\nabla u|^{q-1}{\rm div}Y-|\nabla u|^qY\left(\frac{1}{|\nabla u|}\right)\right)\circ\psi_t\cdot J^0\psi_t\\
&=\left(\frac{1}{|\nabla u|}{\rm div}(|\nabla u|^qY)\right)\circ\psi_t\cdot J^0\psi_t\\
&=\left(\frac{\Delta_q u}{|\nabla u|}\right)\circ\psi_t\cdot J^0\psi_t.
\end{align*}
Thus by assumptions and by Lebesgue dominated convergence theorem function $\nu^q_u$ is differentiable at $s$ and the desired equality holds.
\end{proof}

By above proposition we have an important corollary.
\begin{cor}
If $u$ is $q$--harmonic, then $\nu^q_u$ is constant.
\end{cor}

We now derive the differential equation for $\mathcal{F}_u$ to be the critical point of the $p$--modulus functional. Let 
\begin{equation*}
\lambda(t)=\int_a^t \nu^q_u(t)^{1-p},\quad t\in(a,b),
\end{equation*}
where $(a,b)=u(M)$. Then $\lambda '(t)=\nu^q_u(t)^{1-p}$, hence
\begin{equation*}
(|\nabla(\lambda\circ u)|^{q-1})^{\widehat{\,}}=|\lambda '\circ u|^{q-1}\widehat{|\nabla u|^{q-1}}=1.
\end{equation*}
Taking $v=\lambda \circ u$, the extremal function $f_0$ for the $p$--modulus of $\mathcal{F}_u=\mathcal{F}_v$ equals, by \eqref{eq:extfunsubmersion},
\begin{equation*}
f_0=|\nabla v|^{q-1}.
\end{equation*}
Since 
\begin{equation*}
H_{\mathcal{F}_u}=-{\rm div}_{\mathcal{F}_v}\left(\frac{\nabla v}{|\nabla v|}\right)\frac{\nabla v}{|\nabla v|}
=-{\rm div}\left(\frac{\nabla v}{|\nabla v|}\right)\frac{\nabla v}{|\nabla v|},
\end{equation*}
the equation \eqref{eq:criticalpoint} takes the form
\begin{equation*}
\nabla^{\bot}(\log |\nabla v|^{q-1})=-\,{\rm div}\left(\frac{\nabla v}{|\nabla v|}\right)\frac{\nabla v}{|\nabla v|}.
\end{equation*}
Computing the scalar product of both sides with $\frac{\nabla v}{|\nabla v|}$ we get
\begin{equation*}
\frac{\nabla v}{|\nabla v|}(|\nabla v|^{q-1})=-|\nabla v|^{q-1}{\rm div}\left(\frac{\nabla v}{|\nabla v|}\right).
\end{equation*}
Hence, by the formula ${\rm div}(fX)=f{\rm div} X+Xf$, we get
\begin{equation*}
{\rm div}(|\nabla v|^{q-2}\nabla v)=0,
\end{equation*}
thus $v$ is $q$--harmonic.

Returning to the function $u$ we have 
\begin{align*}
0 &=\Delta_q(\lambda\circ u)={\rm div}(|\nabla(\lambda\circ u)|^{q-2}\nabla(\lambda\circ u))\\
&={\rm div}((\lambda '\circ u)^{q-1}|\nabla u|^{q-2}\nabla u)\\
&=(\lambda '\circ u)^{q-1}\Delta_q u+|\nabla u|^{q-2}\nabla u((\lambda '\circ u)^{q-1})\\
&=(\lambda '\circ u)^{q-2}\left( (\lambda '\circ u)\Delta_q u+(q-1)(\lambda ''\circ u)|\nabla u|^q \right),
\end{align*}
assuming $\lambda$ is twice differentiable. Notice that by Proposition \ref{prop:diffnu}
\begin{equation*}
\lambda ''\circ u=(1-p)(\nu^q_u\circ u)^{-p}\cdot((\nu^q_u)'\circ u)
=(1-p)(\nu^q_u\circ u)^{-p}\widehat{\left(\frac{\Delta_q u}{|\nabla u|}\right)}.
\end{equation*} 
Since
\begin{equation*}
\nu^q_u\circ u=\widehat{|\nabla u|^{q-1}}\quad\textrm{and}\quad \lambda '\circ u=\widehat{|\nabla u|^{q-1}}^{1-p},
\end{equation*}
we obtain
\begin{equation*}
\Delta_q u=(1-q)\frac{\lambda ''\circ u}{\lambda '\circ u}|\nabla u|^q
=\frac{|\nabla u|^q}{\widehat{|\nabla u|^{q-1}}}\widehat{\left(\frac{\Delta_q u}{|\nabla u|}\right)}
=f_0|\nabla u|\widehat{\left(\frac{\Delta_q u}{|\nabla u|}\right)}.
\end{equation*}

We collect obtained results in the following theorem.
\begin{thm}\label{thm:criticalpointfu}
The following two characterizations of critical points of $p$--modulus functional for foliations by level sets of functions hold.
\begin{enumerate}
\item Assume $\nu^q_u$ is differentiable and that \eqref{eq:diffnu} holds. Then, $\mathcal{F}_u$ is a critical point of $p$--modulus functional if and only if $u$ satisfies the differential equation
\begin{equation*}
\frac{\Delta_q u}{|\nabla u|}=f_0\widehat{\left(\frac{\Delta_q u}{|\nabla u|}\right)},
\end{equation*}
where $f_0$ is the extremal function for the $p$--modulus of $\mathcal{F}_u$. 
\item $\mathcal{F}_u$ is a critical point of $p$--modulus functional if and only if there exist $q$--harmonic function $v$ such that $\mathcal{F}_v=\mathcal{F}_u$.
\end{enumerate}
\end{thm}

\begin{cor}\label{cor:criticaldist}
Assume $\mathcal{F}_u$ is given by the distance function $u={\rm dist}(\cdot,L)$ from the codimension one submanifold $L$. Then $\mathcal{F}_u$ is a critical point of the $p$--modulus functional for any $p>1$.
\end{cor}
\begin{proof}
It is known that $|\nabla u|=1$ \cite{rf,dz}. Since the foliation $\mathcal{F}^{\bot}$ orthogonal to $\mathcal{F}$ is by geodesics, it follows that $H_{\mathcal{F}^{\bot}}=0$. Thus extremal function is constant on the leaves $u^{-1}(t)$, i.e. 
\begin{equation*}
f_0=f_0(t)=\frac{1}{\hat{1}}=\frac{1}{\mu_{u^{-1}(t)}(u^{-1}(t))}.
\end{equation*}
By Theorem \ref{thm:criticalpointfu}, $\mathcal{F}_u$ is a critical point of the $p$--modulus functional if and only if
\begin{equation}\label{eq:corcritical}
\hat{1}\cdot\Delta_q u=\widehat{\Delta_q u}.
\end{equation}
Notice, that the $q$--laplacian of $u$ is just the laplacian $\Delta=\Delta_2$ of $u$. Since $u$ is constant on the level sets, $\Delta u$ depends only on $t$, hence is constant on $u^{-1}(t)$. Therefore equation \eqref{eq:corcritical} is trivially satisfied. 
\end{proof}

\section{Final remarks}

Notice the correlation between $p$--stable codimension one transversally orientable foliation and (weighted) Hardy type inequalities.

\begin{thm}\label{thm:Hardttypeineq}
Assume $\mathcal{F}$ is a codimension one $p$--stable transversally orientable foliation on a Riemannian manifold $(M,g)$. Let 
\begin{equation*}
\rho_p=p|\Pi|^2+p{\rm Ric}(N)+q|\nabla_NN|,
\end{equation*}
where $N$ is the unit normal vector field to $\mathcal{F}$. Then the following inequality holds
\begin{equation}\label{eq:Hardytypeineq}
\int_M f^2\rho_p\,f_0^p d\bar{\mu}\leq \int_M |\nabla_{p,q}f|^2\, f_0^pd\bar{\mu},\quad f\in C^{\infty}_0(M),
\end{equation}
where $f_0$ is the extremal function for the $p$--modulus of $\mathcal{F}$.
\end{thm} 
\begin{proof}
Follows immediately by Theorem \ref{thm:codim1}. 
\end{proof}

Notice that in the case $p=q=2$ the gradient $\nabla_{2,2}$ is just the multiple of the usual gradient, more precisely $\nabla_{2,2}=\sqrt{2}\nabla$, and the Hardy type inequality \eqref{eq:Hardytypeineq} is of the form
\begin{equation*}
\int_M f^2\rho_2\,f_0^2d\bar{\mu}\leq 2\int_M |\nabla f|^2\,f_0^2d\bar{\mu}.
\end{equation*}

The existence of $p$--stable foliations given by the level sets of smooth functions follows by well known results in the literature. Let us recall some facts about capacity of a condenser. 

Let $M$ be a Riemannian manifold, $L_0,L_1$ two disjoint closed sets (in our case $L_0$ and $L_1$ are two codimension one submanifolds being the two boundary components of $M$). The triple $(M,L_0,L_1)$ is called a {\it condenser}. Let $\mathcal{A}_q$ be the class of continuous functions $u$ on $M$ which admit $L^q$ integrable gradient and such that $u=0$ on $L_0$ and $u=1$ on $L_1$. The $q$--{\it capacity} of $(M,L_0,L_1)$ is the number
\begin{equation*}
{\rm cap}_q(M,L_0,L_1)=\inf_{u\in \mathcal{A}_q}\int_M |\nabla u|^q\,d\bar{\mu}.
\end{equation*}  
One can show \cite{jf} that if $L_0$ and $L_0$ are compact and do not reduce to a point, $\mathcal{A}_q$ is nonempty, then there is unique function $u$ which realizes the $q$--capacity. Moreover, slightly generalizing the considerations in \cite{wz} or with the use of primitive functions in \cite{bf}, see also \cite{jf0}, the $q$--capacity equals the $p$--modulus of the foliation $\mathcal{F}_u$ by level sets of $u$. Thus $\mathcal{F}_u$ is $p$--stable.

\end{document}